\documentclass[a4paper,american,11pt]{amsart}

\usepackage[T1]{fontenc}
\usepackage{lmodern}
\usepackage{amsmath,amssymb,mathtools}
\usepackage{enumitem}
\usepackage[hidelinks]{hyperref}
\usepackage[backend=biber,style=numeric,sorting=nyt,maxbibnames=10]{biblatex}
\usepackage{array}
\usepackage{tabularx}
\usepackage{booktabs}
\numberwithin{equation}{section}

\newtheorem{theorem}{Theorem}[section]
\newtheorem{proposition}[theorem]{Proposition}
\newtheorem{lemma}[theorem]{Lemma}
\newtheorem{corollary}[theorem]{Corollary}
\theoremstyle{definition}
\newtheorem{definition}[theorem]{Definition}

\newtheorem{problem}[theorem]{Problem}
\newtheorem{conjecture}[theorem]{Conjecture}
\newcommand{\Z}{\mathbb Z}
\newcommand{\F}{\mathbb F}

\newcommand{\Op}{\mathcal O_p}
\newcommand{\Rect}{\mathcal R}
\newcommand{\Prect}{\mathcal P}
\newcommand{\vp}{v_p}
\DeclareMathOperator{\slf}{sl}
\begin{document}

\title[Ljunggren--Jacobsthal and Bailey-type congruences]
{Ljunggren--Jacobsthal and Bailey-Type Congruences for Rectangular Gaussian  Binomial Coefficients}

\author{Kevin Calderon}
\address{Departamento de Matem\'aticas, Facultad de Ciencias, Universidad Nacional Aut\'onoma de M\'exico, 04510 Ciudad de M\'exico, Mexico}
\email{kevincalderon@ciencias.unam.mx}

\begin{abstract}
Motivated by the polynomial form of Kalinin's Gaussian analogue of Wolstenholme's theorem, following Kalinin, we study a two-dimensional factorial ratio over the Gaussian integers, which we call the rectangular Gaussian binomial coefficient. For a rational prime $p\equiv 3\pmod 4$, we first prove that these coefficients are $p$-integral. Our main result is a Ljunggren--Jacobsthal-type supercongruence: for $p>5$ and every $k\geq 1$, simultaneous dilation of all four parameters by $p^k$ changes the coefficient by a multiple of $p^{3k}$. In particular, this proves the inert-prime case of a conjecture of Kalinin. We also establish a rectangular Bailey-type congruence modulo $p$ for parameters consisting of a large $p$-multiple and one base-$p$ digit. Its shape parallels Bailey's prime-power refinements of Lucas's theorem, but two additional ordinary binomial factors occur, reflecting the vertical and horizontal boundary strips of a rectangular block decomposition. The proofs combine reciprocal-power-sum estimates in $\Z_p[i]$ with factorizations of rectangular products into complete $p^k\times p^k$ blocks.
\end{abstract}


\maketitle

\section{Introduction}

\subsection*{Motivation}

For integers $N\geq K\geq 0$, the binomial coefficient can be written as the quotient
\begin{equation}\label{eq:ordinary-binomial-product}
\binom{N}{K}
=
\frac{N(N-1)\cdots(N-K+1)}{1\cdot2\cdots K}
=
\frac{\displaystyle\prod_{x=N-K+1}^{N}x}
{\displaystyle\prod_{x=1}^{K}x}.
\end{equation}
One of the most striking arithmetic properties of these coefficients is Wolstenholme's congruence
\begin{equation}\label{eq:wolstenholme-classical}
\binom{2p-1}{p-1}\equiv 1\pmod{p^3},
\qquad p\geq5,
\end{equation}
for every prime $p$ \cite{Wolstenholme1862}. Equivalently,
\begin{equation}\label{eq:wolstenholme-product}
(2p-1)(2p-2)\cdots(p+1)
\equiv
(p-1)!
\pmod{p^3}.
\end{equation}
This equivalence is especially transparent from the polynomial
\[
f_p(X):=\prod_{n=1}^{p-1}(X-n).
\]
Since $p-1$ is even,
\begin{equation}\label{eq:wolstenholme-polynomial-quotient}
\frac{f_p(2p)}{f_p(0)}
=
\frac{\displaystyle\prod_{n=1}^{p-1}(2p-n)}
{\displaystyle\prod_{n=1}^{p-1}n}
=
\binom{2p-1}{p-1}.
\end{equation}
Consequently, the product congruence $f_p(2p)\equiv f_p(0)\pmod{p^3}$ becomes Wolstenholme's binomial congruence after division by the basic factorial product.

Kalinin's Gaussian version of Wolstenholme's theorem suggests that the same mechanism should persist in the lattice of Gaussian integers \cite{Kalinin2025,KalininZottor2026Gaussian}. Let
\begin{equation}\label{eq:admissible-gaussian-block}
\mathcal U_p
:=
\left\{
 n+mi:
 1\leq n,m\leq p-1,
 \quad p\nmid n^2+m^2
\right\},
\end{equation}
and write $r:=|\mathcal U_p|$. Consider the polynomial
\begin{equation}\label{eq:gaussian-root-polynomial}
g_p(Z)
:=
\prod_{\alpha\in\mathcal U_p}(Z-\alpha)
=
Z^r+a_1Z^{r-1}+\cdots+a_{r-1}Z+a_r.
\end{equation}
Its constant coefficient is
\[
a_r=(-1)^r\prod_{\alpha\in\mathcal U_p}\alpha.
\]
Every $\alpha\in\mathcal U_p$ is nonzero modulo $p$, so $a_r$ is a $p$-adic unit. The coefficients next to the constant term are controlled by symmetric functions of the reciprocal roots. Indeed, Vieta's formulas give
\begin{equation}\label{eq:vieta-reciprocals}
\frac{a_{r-k}}{a_r}
=
(-1)^k e_k\bigl(\alpha^{-1}:\alpha\in\mathcal U_p\bigr),
\end{equation}
where $e_k$ denotes the $k$-th elementary symmetric polynomial.

Introduce the reciprocal power sums
\begin{equation}\label{eq:kalinin-power-sums}
S_p^{(j)}
:=
\sum_{\alpha\in\mathcal U_p}\frac{1}{\alpha^j}.
\end{equation}
Kalinin's Gaussian Wolstenholme theorem yields, for $p>5$,
\begin{equation}\label{eq:kalinin-valuations}
\begin{aligned}
S_p^{(1)}&\equiv0\pmod{p^4},
&\qquad S_p^{(2)}&\equiv0\pmod{p^3},\\
S_p^{(3)}&\equiv0\pmod{p^2},
& S_p^{(4)}&\equiv0\pmod p.
\end{aligned}
\end{equation}
Equivalently,
\[
\vp\bigl(S_p^{(j)}\bigr)\geq5-j,
\qquad 1\leq j\leq4.
\]
Newton's identities read
\begin{equation}\label{eq:newton-introduction}
k e_k
=
\sum_{j=1}^{k}(-1)^{j-1}e_{k-j}S_p^{(j)},
\qquad e_0=1.
\end{equation}
Because $1,2,3,4$ are invertible modulo $p$ when $p>5$, equations \eqref{eq:kalinin-valuations} and \eqref{eq:newton-introduction} imply inductively that
\begin{equation}\label{eq:elementary-symmetric-valuations}
e_k\bigl(\alpha^{-1}:\alpha\in\mathcal U_p\bigr)
\equiv0\pmod{p^{5-k}},
\qquad 1\leq k\leq4.
\end{equation}
Since $a_r$ is a unit, \eqref{eq:vieta-reciprocals} consequently gives
\begin{equation}\label{eq:coefficient-valuations}
a_{r-k}\equiv0\pmod{p^{5-k}},
\qquad 1\leq k\leq4.
\end{equation}

Now let $z=A+iB\in\Z[i]$ and evaluate \eqref{eq:gaussian-root-polynomial} at $Z=pz$. Expanding from the constant term gives
\[
g_p(pz)
=
a_r+a_{r-1}pz+a_{r-2}p^2z^2+a_{r-3}p^3z^3+a_{r-4}p^4z^4+\cdots.
\]
For $1\leq k\leq4$, the term $a_{r-k}p^kz^k$ is divisible by $p^{5-k}p^k=p^5$, while every term of degree at least five is automatically divisible by $p^5$. Therefore
\begin{equation}\label{eq:gaussian-polynomial-congruence}
g_p(pA+ipB)\equiv g_p(0)\pmod{p^5}.
\end{equation}
When $p\equiv3\pmod4$, the congruence $n^2+m^2\equiv0\pmod p$ forces $n\equiv m\equiv0\pmod p$. Hence $\mathcal U_p$ is the entire square $1\leq n,m\leq p-1$, and $r=(p-1)^2$ is even. Thus \eqref{eq:gaussian-polynomial-congruence} becomes the product congruence
\begin{equation}\label{eq:kalinin-product}
\prod_{n=1}^{p-1}\prod_{m=1}^{p-1}
\bigl(pA+ipB-(n+mi)\bigr)
\equiv
\prod_{n=1}^{p-1}\prod_{m=1}^{p-1}(n+mi)
\pmod{p^5}.
\end{equation}

In the ordinary setting, the quotient of the two products in \eqref{eq:wolstenholme-product} is a binomial coefficient. It is therefore natural to seek a two-dimensional factorial ratio that packages \eqref{eq:kalinin-product} in the same way. This leads to the following definition.

\begin{definition}\label{def:rectangular-coefficient}
For integers
\[
A\geq C\geq1,
\qquad
B\geq D\geq1,
\]
the \emph{rectangular Gaussian binomial coefficient} is
\begin{equation}\label{eq:def-R}
\Rect(A,B;C,D)
:=
\frac{\displaystyle
\prod_{n=0}^{C-1}\prod_{m=0}^{D-1}
\bigl(A+Bi-(n+mi)\bigr)}
{\displaystyle
\prod_{n=1}^{C}\prod_{m=1}^{D}(n+mi)}.
\end{equation}
\end{definition}

Equivalently,
\begin{equation}\label{eq:def-R-rectangle}
\Rect(A,B;C,D)
=
\frac{\displaystyle
\prod_{x=A-C+1}^{A}\prod_{y=B-D+1}^{B}(x+iy)}
{\displaystyle
\prod_{x=1}^{C}\prod_{y=1}^{D}(x+iy)}.
\end{equation}
The numerator is the product over a translated $C\times D$ lattice rectangle, while the denominator is the product over the basic $C\times D$ rectangle. Thus \eqref{eq:def-R-rectangle} is the direct two-dimensional counterpart of the interval quotient \eqref{eq:ordinary-binomial-product}. The word ``Gaussian'' refers to the lattice $\Z[i]$; these coefficients are unrelated to the usual $q$-binomial coefficients, which are also called Gaussian binomial coefficients. Moreover, they need not lie in $\Z[i]$, so their congruences must be interpreted locally.

The connection with Kalinin's product is exact. Setting $C=D=p-1$ and replacing the two upper parameters by $pA-1$ and $pB-1$, respectively, gives
\begin{align}\label{eq:gaussian-product-quotient}
\Rect(pA-1,pB-1;p-1,p-1)
&=
\frac{\displaystyle
\prod_{n=1}^{p-1}\prod_{m=1}^{p-1}
\bigl(pA+ipB-(n+mi)\bigr)}
{\displaystyle
\prod_{n=1}^{p-1}\prod_{m=1}^{p-1}(n+mi)}.
\end{align}
Consequently, \eqref{eq:kalinin-product} is precisely the binomial-type congruence
\begin{equation}\label{eq:kalinin-rectangular-special-case}
\Rect(pA-1,pB-1;p-1,p-1)\equiv1\pmod{p^5}.
\end{equation}
This identity explains where the new coefficient comes from: it is the Gaussian rectangular factorial ratio naturally produced by the polynomial form of Kalinin's theorem.

\subsection*{The two motivating families of congruences}

The first classical model is the Ljunggren--Jacobsthal congruence. For a prime $p\geq5$, it implies in particular that
\begin{equation}\label{eq:classical-iterated}
\binom{p^kA}{p^kC}
\equiv
\binom{p^{k-1}A}{p^{k-1}C}
\pmod{p^{3k}},
\qquad k\geq1;
\end{equation}
see \cite{BrunEtAl1949,Straub2011,Rowland2022}. Kalinin conjectured the rectangular analogue of the case $k=1$ \cite[Conjecture~3]{Kalinin2025}. Our first theorem proves the full iterated form for primes inert in $\Z[i]$.

\begin{theorem}\label{thm:ljunggren}
Let $p>5$ be a prime with $p\equiv3\pmod4$, and let $k\geq1$. For all integers $A\geq C\geq1$ and $B\geq D\geq1$,
\begin{equation}\label{eq:main-congruence}
\Rect(p^kA,p^kB;p^kC,p^kD)
\equiv
\Rect(p^{k-1}A,p^{k-1}B;p^{k-1}C,p^{k-1}D)
\pmod{p^{3k}}.
\end{equation}
\end{theorem}

The second model is Bailey's refinement of the one-digit form of Lucas's theorem. Lucas's congruence is
\[
\binom{pA+\alpha}{pC+\gamma}
\equiv
\binom{A}{C}\binom{\alpha}{\gamma}
\pmod p,
\qquad 0\leq\gamma\leq\alpha<p,
\]
whereas Bailey showed that analogous coarse-block/digit factorizations persist modulo higher powers of $p$ \cite{Bailey1990,Bailey1992}; see also \cite{DavisWebb1990,Granville1997,Rowland2022}. This is the source of our second theorem. Its modulus is only $p$, but its decomposition into a large $p$-block and a residual digit rectangle is Bailey-type. Two additional ordinary binomial factors appear because, in two dimensions, the residual rectangle produces vertical and horizontal boundary strips.

\begin{theorem}\label{thm:bailey}
Let $p\equiv3\pmod4$ be an odd prime. Let $A\geq C\geq1$ and $B\geq D\geq1$, and suppose
\[
1\leq\gamma\leq\alpha\leq p-1,
\qquad
1\leq\delta\leq\beta\leq p-1.
\]
Then
\begin{align}\label{eq:bailey-congruence}
&\Rect(pA+\alpha,pB+\beta;pC+\gamma,pD+\delta)\notag\\
&\quad\equiv
\Rect(A,B;C,D)\,
\Rect(\alpha,\beta;\gamma,\delta)
\binom{\alpha}{\gamma}^{D}
\binom{\beta}{\delta}^{C}
\pmod p.
\end{align}
\end{theorem}

\subsection*{Organization of the paper}
Section~\ref{sec:preliminaries} fixes the local meaning of congruence, introduces rectangular factorials, and proves $p$-integrality. Section~\ref{sec:ljunggren-proof} establishes reciprocal-power-sum estimates on a complete $p^k\times p^k$ block and uses them to prove Theorem~\ref{thm:ljunggren}. Section~\ref{sec:bailey-proof} evaluates complete blocks and boundary strips modulo $p$, yielding Theorem~\ref{thm:bailey}. Section~\ref{sec:consequences} records immediate consequences, Gamma-function representations, and open questions concerning analogues of classical binomial identities.

\section{ Preliminaries}\label{sec:preliminaries}

Fix throughout this section an odd prime $p\equiv 3\pmod 4$. Then $T^2+1$ is irreducible over $\F_p$, so $p$ is inert in $\Z[i]$ and
\[
\Op=\Z_p[i]
\]
is the ring of integers in the unramified quadratic extension $\mathbb Q_p(i)/\mathbb Q_p$. For $X,Y\in\Op$ and $r\geq 0$, we write
\[
X\equiv Y\pmod{p^r}
\quad\Longleftrightarrow\quad
X-Y\in p^r\Op.
\]
For $x,y\in\Z$, not both zero, inertness gives
\begin{equation}\label{eq:valuation-coordinatewise}
\vp(x+iy)=\min\{\vp(x),\vp(y)\}
=\sum_{r\geq 1}\mathbf 1_{\{p^r\mid x,\;p^r\mid y\}}.
\end{equation}

For nonnegative integers $X,Y$, define the rectangular factorial
\begin{equation}\label{eq:def-P}
\Prect(X,Y):=\prod_{x=1}^{X}\prod_{y=1}^{Y}(x+iy),
\end{equation}
with $\Prect(0,Y)=\Prect(X,0)=1$. Multiplicative inclusion--exclusion gives
\begin{equation}\label{eq:R-via-P}
\Rect(A,B;C,D)
=
\frac{\Prect(A,B)\Prect(A-C,B-D)}
{\Prect(A-C,B)\Prect(A,B-D)\Prect(C,D)}.
\end{equation}

\begin{lemma}\label{lem:integrality}
For every $A\geq C\geq 1$ and $B\geq D\geq 1$,
\[
\Rect(A,B;C,D)\in\Op.
\]
\end{lemma}

\begin{proof}
For $r\geq 1$, let $N_r(A,C)$ denote the number of multiples of $p^r$ in the interval $[A-C+1,A]$. Every interval of $C$ consecutive integers contains at least $\lfloor C/p^r\rfloor$ multiples of $p^r$, and hence
\[
N_r(A,C)\geq \left\lfloor\frac{C}{p^r}\right\rfloor.
\]
The analogous inequality holds for $N_r(B,D)$. By \eqref{eq:valuation-coordinatewise}, the valuation of the numerator in \eqref{eq:def-R-rectangle} is
\[
\sum_{r\geq 1}N_r(A,C)N_r(B,D),
\]
whereas the valuation of its denominator is
\[
\sum_{r\geq 1}
\left\lfloor\frac{C}{p^r}\right\rfloor
\left\lfloor\frac{D}{p^r}\right\rfloor.
\]
The first quantity is at least the second, proving the claim.
\end{proof}

\section{The Ljunggren--Jacobsthal-type congruence}\label{sec:ljunggren-proof}

For $k\geq 1$, let
\begin{equation}\label{eq:def-Uk}
\mathcal U_k
:=
\left\{x+iy:1\leq x,y\leq p^k,
\text{ and }p\nmid(x,y)\right\},
\end{equation}
where $p\nmid(x,y)$ means that $x$ and $y$ are not simultaneously divisible by $p$. Thus $\mathcal U_k$ is a complete set of representatives for $(\Z[i]/p^k\Z[i])^{\times}$. For $j\geq 1$, put
\begin{equation}\label{eq:def-Hjk}
H_j(k):=\sum_{z\in\mathcal U_k}\frac{1}{z^j}\in\Op.
\end{equation}

\begin{lemma}\label{lem:power-sums}
Let $p>5$ and $p\equiv 3\pmod 4$. For every $k\geq 1$,
\begin{equation}\label{eq:H-valuations}
H_1(k)\equiv 0\pmod{p^{2k}},
\qquad
H_2(k)\equiv 0\pmod{p^k}.
\end{equation}
\end{lemma}

\begin{proof}
We first consider $k=1$. Since $p\equiv 3\pmod 4$, the congruence $x^2+y^2\equiv 0\pmod p$ forces $x\equiv y\equiv 0\pmod p$. Therefore
\begin{align*}
H_1(1)
&=
\sum_{x=1}^{p-1}\sum_{y=1}^{p-1}\frac{1}{x+iy}
+
\sum_{x=1}^{p-1}\frac{1}{x+ip}
+
\sum_{y=1}^{p-1}\frac{1}{p+iy}.
\end{align*}
Kalinin's Gaussian Wolstenholme theorem gives
\[
\sum_{x=1}^{p-1}\sum_{y=1}^{p-1}\frac{1}{x+iy}
\equiv 0\pmod{p^4}
\]
\cite[Theorem~1]{Kalinin2025}. For $1\leq x\leq p-1$,
\[
\frac{1}{x+ip}
\equiv
\frac1x-\frac{ip}{x^2}
\pmod{p^2}.
\]
The classical harmonic congruences
\[
\sum_{x=1}^{p-1}\frac1x\equiv0\pmod{p^2},
\qquad
\sum_{x=1}^{p-1}\frac1{x^2}\equiv0\pmod p
\]
therefore show that the second sum vanishes modulo $p^2$; the third is treated identically. Hence $H_1(1)\equiv0\pmod{p^2}$.

Reduction modulo $p$ identifies $\mathcal U_1$ with $\F_{p^2}^{\times}$. Since inversion permutes this group,
\[
H_2(1)
\equiv
\sum_{z\in\F_{p^2}^{\times}}z^{-2}
=
\sum_{z\in\F_{p^2}^{\times}}z^2
=0
\pmod p.
\]
This proves the base case.

Assume now that $k\geq 2$. Every $z\in\mathcal U_k$ has a unique expression
\[
z=u+p^{k-1}w,
\qquad
u\in\mathcal U_{k-1},
\qquad
w=a+ib,
\quad 0\leq a,b\leq p-1.
\]
Set
\[
M_\ell:=\sum_{a,b=0}^{p-1}(a+ib)^\ell.
\]
A direct calculation gives
\[
M_0=p^2,
\qquad
\vp(M_1),\vp(M_2),\vp(M_3)\geq 2.
\]
For example,
\[
M_1=\frac{p^2(p-1)}2(1+i),
\qquad
M_2=\frac{i p^2(p-1)^2}{2}.
\]
Expanding $p$-adically and summing over $u$ and $w$ yields
\begin{equation}\label{eq:H-recursion}
H_j(k)
=
\sum_{\ell\geq 0}
(-1)^\ell
\binom{j+\ell-1}{\ell}
 p^{(k-1)\ell}M_\ell H_{j+\ell}(k-1).
\end{equation}

For $j=2$, the term $\ell=0$ has valuation at least $2+(k-1)$. The term $\ell=1$ has valuation at least $(k-1)+2$, and every term with $\ell\geq2$ has valuation at least $2(k-1)\geq k$. Thus $\vp(H_2(k))\geq k$.

For $j=1$, the terms $\ell=0,1,2$ have valuations at least
\[
2+2(k-1),
\qquad
(k-1)+2+(k-1),
\qquad
2(k-1)+2,
\]
respectively, and hence are divisible by $p^{2k}$. If $k\geq3$, every term with $\ell\geq3$ contains $p^{\ell(k-1)}$ with $\ell(k-1)\geq2k$. If $k=2$, the term $\ell=3$ is divisible by $p^{3}M_3\subseteq p^5\Op$, while the terms with $\ell\geq4$ are automatically divisible by $p^4$. Consequently $\vp(H_1(k))\geq2k$, completing the induction.
\end{proof}

\begin{lemma}\label{lem:block-translation}
Let
\[
G_k:=\prod_{z\in\mathcal U_k}z.
\]
For every $t\in\Z[i]$,
\begin{equation}\label{eq:block-translation}
\prod_{z\in\mathcal U_k}(z+p^kt)
\equiv
G_k
\pmod{p^{3k}}.
\end{equation}
\end{lemma}

\begin{proof}
Because every $z\in\mathcal U_k$ is a unit in $\Op$,
\[
\frac{1}{G_k}\prod_{z\in\mathcal U_k}(z+p^kt)
=
\prod_{z\in\mathcal U_k}
\left(1+\frac{p^kt}{z}\right).
\]
Let $e_r$ be the $r$th elementary symmetric polynomial in the family $\{z^{-1}:z\in\mathcal U_k\}$. Then
\[
e_1=H_1(k),
\qquad
2e_2=H_1(k)^2-H_2(k).
\]
Lemma~\ref{lem:power-sums} gives $e_1\in p^{2k}\Op$ and $e_2\in p^k\Op$. Therefore
\begin{align*}
\prod_{z\in\mathcal U_k}
\left(1+\frac{p^kt}{z}\right)
&=
1+p^kt e_1+p^{2k}t^2e_2
+
\sum_{r\geq3}p^{kr}t^re_r
\equiv 1\pmod{p^{3k}},
\end{align*}
which proves \eqref{eq:block-translation}.
\end{proof}

For nonnegative integers $X,Y$, define the unit-factor product
\begin{equation}\label{eq:def-Fk}
F_k(X,Y)
:=
\prod_{\substack{1\leq x\leq p^kX\\1\leq y\leq p^kY\\p\nmid(x,y)}}(x+iy),
\end{equation}
with $F_k(0,Y)=F_k(X,0)=1$.

\begin{lemma}\label{lem:complete-rectangles}
For all $X,Y\geq0$,
\begin{equation}\label{eq:Fk-congruence}
F_k(X,Y)\equiv G_k^{XY}\pmod{p^{3k}}.
\end{equation}
\end{lemma}

\begin{proof}
Partition $[1,p^kX]\times[1,p^kY]$ into $XY$ blocks of size $p^k\times p^k$. The block indexed by $0\leq r<X$ and $0\leq s<Y$ has unit-factor product
\[
\prod_{z\in\mathcal U_k}\bigl(z+p^k(r+is)\bigr),
\]
which is congruent to $G_k$ modulo $p^{3k}$ by Lemma~\ref{lem:block-translation}. Multiplication over all blocks gives \eqref{eq:Fk-congruence}.
\end{proof}

\begin{proof}[Proof of Theorem~\ref{thm:ljunggren}]
Set
\[
R_k:=\Rect(p^kA,p^kB;p^kC,p^kD).
\]
Since $p$ is inert,
\[
p\mid x+iy
\quad\Longleftrightarrow\quad
p\mid x\text{ and }p\mid y.
\]
Separate in the numerator and denominator of $R_k$ the factors divisible by $p$. Dividing both coordinates of each such factor by $p$ produces precisely the numerator and denominator of $R_{k-1}$, and the same number of scalar factors $p$ occurs above and below. Hence
\begin{equation}\label{eq:R-ratio-units}
\frac{R_k}{R_{k-1}}
=
\frac{F_k(A,B)F_k(A-C,B-D)}
{F_k(A-C,B)F_k(A,B-D)F_k(C,D)}.
\end{equation}
The right-hand side is a unit of $\Op$. Applying Lemma~\ref{lem:complete-rectangles} to each factor gives
\[
\frac{R_k}{R_{k-1}}
\equiv
G_k^E
\pmod{p^{3k}},
\]
where
\begin{align*}
E
&=AB+(A-C)(B-D)-(A-C)B-A(B-D)-CD
=0.
\end{align*}
Thus $R_k/R_{k-1}\equiv1\pmod{p^{3k}}$. By Lemma~\ref{lem:integrality}, $R_{k-1}\in\Op$, so
\[
R_k-R_{k-1}
=R_{k-1}\left(\frac{R_k}{R_{k-1}}-1\right)
\in p^{3k}\Op.
\]
This is exactly \eqref{eq:main-congruence}.
\end{proof}

\section{The Bailey-type congruence}\label{sec:bailey-proof}

For nonnegative integers $M,N$ and digits $0\leq a,b\leq p-1$, define
\begin{equation}\label{eq:def-unit-product}
\mathcal V_p(M,N;a,b)
:=
\prod_{\substack{1\leq x\leq pM+a\\1\leq y\leq pN+b\\p\nmid(x,y)}}(x+iy).
\end{equation}
The nonunit factors of $\Prect(pM+a,pN+b)$ are exactly the numbers $p(x+iy)$ with $1\leq x\leq M$ and $1\leq y\leq N$. Therefore
\begin{equation}\label{eq:P-splitting}
\Prect(pM+a,pN+b)
=p^{MN}\Prect(M,N)\mathcal V_p(M,N;a,b).
\end{equation}

\begin{lemma}\label{lem:unit-product-mod-p}
For $0\leq a,b\leq p-1$,
\begin{equation}\label{eq:unit-product-mod-p}
\mathcal V_p(M,N;a,b)
\equiv
(-1)^{MN}X_a^N Y_b^M\Prect(a,b)
\pmod p,
\end{equation}
where
\[
X_a:=2^a a!,
\qquad
Y_b:=(-2i)^b b!.
\]
\end{lemma}

\begin{proof}
Partition the rectangle $[1,pM+a]\times[1,pN+b]$ into $MN$ complete $p\times p$ blocks, $N$ vertical boundary strips of width $a$, $M$ horizontal boundary strips of height $b$, and one $a\times b$ corner.

The unit factors in each complete block reduce to all elements of $\F_{p^2}^{\times}$, whose product is $-1$. For $x,y\in\F_p^{\times}$, the identity $T^p-T=\prod_{t\in\F_p}(T-t)$ and the relation $i^p=-i$ give
\begin{equation}\label{eq:strip-products}
\prod_{t\in\F_p}(x+it)=2x,
\qquad
\prod_{t\in\F_p}(t+iy)=-2iy.
\end{equation}
Thus the vertical strips contribute $X_a^N$, the horizontal strips contribute $Y_b^M$, and the corner contributes $\Prect(a,b)$. Multiplying these contributions proves \eqref{eq:unit-product-mod-p}.
\end{proof}

\begin{proof}[Proof of Theorem~\ref{thm:bailey}]
By \eqref{eq:R-via-P},
\begin{align}\label{eq:bailey-five-P}
&\Rect(pA+\alpha,pB+\beta;pC+\gamma,pD+\delta)\notag\\
&=
\frac{\Prect(pA+\alpha,pB+\beta)}
{\Prect(p(A-C)+\alpha-\gamma,pB+\beta)}\notag\\
&\qquad\times
\frac{\Prect(p(A-C)+\alpha-\gamma,p(B-D)+\beta-\delta)}
{\Prect(pA+\alpha,p(B-D)+\beta-\delta)\Prect(pC+\gamma,pD+\delta)}.
\end{align}
Apply \eqref{eq:P-splitting} to all five factors. The total exponent of $p$ is
\[
AB+(A-C)(B-D)-(A-C)B-A(B-D)-CD=0,
\]
so all powers of $p$ cancel. The coarse rectangular factorials combine to $\Rect(A,B;C,D)$, while the corner factors combine to $\Rect(\alpha,\beta;\gamma,\delta)$.

It remains to evaluate the complete-block signs and strip factors using Lemma~\ref{lem:unit-product-mod-p}. The signs cancel because their total exponent is again zero. The vertical-strip factors contribute
\begin{align*}
\frac{X_\alpha^B X_{\alpha-\gamma}^{B-D}}
{X_{\alpha-\gamma}^{B}X_\alpha^{B-D}X_\gamma^D}
&=
\left(\frac{X_\alpha}{X_{\alpha-\gamma}X_\gamma}\right)^D
=
\binom{\alpha}{\gamma}^{D},
\end{align*}
and the horizontal-strip factors contribute
\begin{align*}
\frac{Y_\beta^A Y_{\beta-\delta}^{A-C}}
{Y_\beta^{A-C}Y_{\beta-\delta}^{A}Y_\delta^C}
&=
\left(\frac{Y_\beta}{Y_{\beta-\delta}Y_\delta}\right)^C
=
\binom{\beta}{\delta}^{C}.
\end{align*}
Combining all contributions yields \eqref{eq:bailey-congruence}.
\end{proof}

\section{Conceptual explanation}\label{sec:conceptual}

The arithmetic developed above belongs to a classical--Gaussian dictionary \cite{Katz1977,Coffey2018,Hurwitz1899Lemniskatische}  . Each row of Table~\ref{tab:dictionary} records one stage of the same passage: an ambient lattice determines a symmetry; the symmetry selects the reciprocal moments that can survive; those moments occur in the expansions of natural special functions; their finite analogues govern translated products; and the resulting product identities become congruences for factorial ratios. In the Gaussian column, the lemniscatic functions add a further layer: the square lattice is the period lattice of an elliptic curve with complex multiplication by $\Z[i]$, and its formal group, torsion points, and reductions modulo primes organize the same reciprocal-power sums that enter the block arguments above. The purpose of this section is to explain the two columns separately and in the order in which the rows appear.

\begin{table}[ht]
\centering
\small
\renewcommand{\arraystretch}{1.14}
\setlength{\tabcolsep}{6pt}
\begin{tabularx}{\textwidth}{>{\raggedright\arraybackslash}X!{\vrule width .35pt}>{\raggedright\arraybackslash}X}
\toprule
\textbf{Classical integer theory} & \textbf{Gaussian integer theory}\\
\midrule
The lattice $\Z$ & The square lattice $\Z[i]$\\
Reflection symmetry $m\mapsto-m$ & Quarter-turn symmetry $\alpha\mapsto i\alpha$\\
Circular sine $\sin z$ & Lemniscatic sine $\slf(z)$\\
$\pi\cot(\pi z)$ and $\pi^2\csc^2(\pi z)$ & $\zeta_{\omega\Z[i]}(z)$ and $\wp(z;\omega\Z[i])$\\
$\displaystyle\sum_{m\ne0}m^{-2n}$ & $\displaystyle\sum_{\alpha\ne0}\alpha^{-4n}$\\
Bernoulli numbers $B_{2n}$ & Bernoulli--Hurwitz numbers $E_n$\\
Finite sums $\displaystyle\sum a^{-r}$ over $\F_p^{\times}$ & Finite sums $\displaystyle\sum z^{-r}$ over $\F_{p^2}^{\times}$ for inert $p$\\
The multiplicative group and its formal coordinate & The CM elliptic curve $y^2=x^3-x$ and its formal group\\
Translated intervals of units & Translated square blocks of Gaussian units\\
Ordinary factorial ratios $\binom NK$ & Rectangular ratios $\Rect(A,B;C,D)$\\
Lucas, Bailey, Ljunggren--Jacobsthal congruences & Rectangular Bailey-type and Ljunggren--Jacobsthal congruences\\
\bottomrule
\end{tabularx}
\caption{The classical--Gaussian correspondence.}
\label{tab:dictionary}
\end{table}

\subsection{The classical case}\label{subsec:classical-column}

\emph{The lattice.}
The first row places the classical theory on the one-dimensional lattice $\Z$. Products of consecutive integers, ordinary factorials, and binomial coefficients are all built from finite portions of this lattice. The analytic functions in the subsequent rows are therefore chosen so that their zeros or poles reproduce the same set of integer points.

\emph{Reflection symmetry.}
The map $m\mapsto-m$ preserves $\Z$. Consequently, for every absolutely convergent reciprocal sum,
\[
\sum_{m\in\Z\setminus\{0\}}m^{-r}
=(-1)^r\sum_{m\in\Z\setminus\{0\}}m^{-r}.
\]
Thus the sum vanishes when $r$ is odd. This order-two symmetry explains why the natural global moments in the classical column have exponent $2n$.

\emph{The circular sine.}
After the usual normalization, the zero set of $\sin(\pi z)$ is exactly $\Z$. Its product expansion packages all nonzero lattice points into one analytic object, just as a factorial packages a finite interval of integers into one product.

\emph{The cotangent and cosecant-square functions.}
Taking logarithmic derivatives gives
\[
\frac{d}{dz}\log\sin(\pi z)=\pi\cot(\pi z),
\qquad
-\frac{d}{dz}\bigl(\pi\cot(\pi z)\bigr)=\pi^2\csc^2(\pi z).
\]
The first function has simple poles at the integers, and the second has double poles. Their Laurent expansions therefore convert the geometry of the lattice into reciprocal-power sums.

\emph{Infinite reciprocal moments and Bernoulli numbers.}
The cotangent expansion gives
\[
\pi\cot(\pi z)
=
\frac1z-2\sum_{n\geq1}\zeta(2n)z^{2n-1}.
\]
Equivalently,
\[
\sum_{m\in\Z\setminus\{0\}}\frac1{m^{2n}}
=
(-1)^{n+1}\frac{(2\pi)^{2n}}{(2n)!}B_{2n}.
\]
Thus $B_{2n}$ is the normalized reciprocal moment of order $2n$ of the lattice $\Z$.

\emph{Finite reciprocal sums.}
The finite counterpart is obtained by replacing the nonzero integers with the nonzero residue classes modulo a prime. For $r\geq1$, put
\[
H_r^{\mathrm{cl}}(p)
=
\sum_{a\in\F_p^\times}a^{-r}.
\]
Since $\F_p^\times$ is cyclic of order $p-1$, this sum vanishes modulo $p$ unless $p-1$ divides $r$. Stronger $p$-adic estimates for the first moments are the finite arithmetic form of the cancellations visible in the global lattice sums.

\emph{Translated intervals and factorial ratios.}
If $U_k^{\mathrm{cl}}$ is a complete set of unit representatives modulo $p^k$, then
\[
\frac{\prod_{a\in U_k^{\mathrm{cl}}}(a+p^kt)}
{\prod_{a\in U_k^{\mathrm{cl}}}a}
=
\prod_{a\in U_k^{\mathrm{cl}}}
\left(1+\frac{p^kt}{a}\right).
\]
The coefficients of this product are elementary symmetric functions of the reciprocals $a^{-1}$, and Newton's identities express them through the finite moments. Their divisibility forces translated unit intervals to have the same product modulo a high power of $p$.

An ordinary binomial coefficient is a quotient of interval products. After the factors divisible by $p$ are separated, the remaining factors split into translated intervals of units, and their common block products cancel between numerator and denominator. This is the product-theoretic reason that local moment estimates pass naturally to binomial coefficients and produce the congruences of Lucas, Bailey, and Ljunggren--Jacobsthal.

\subsection{The Gaussian case}\label{subsec:gaussian-column}

\emph{The square lattice.}
The Gaussian column replaces the line $\Z$ with the square lattice $\Z[i]$. A finite rectangular region consists of the numbers $x+iy$ with $(x,y)$ in a rectangle of integer points. The rectangular factorials and the coefficient $\Rect(A,B;C,D)$ are obtained by multiplying over such two-dimensional regions, so $\Z[i]$ is the natural ambient lattice for the present paper.

\emph{Quarter-turn symmetry.}
Multiplication by $i$ preserves $\Z[i]$ and rotates it through a quarter turn. Hence, for an absolutely convergent Gaussian reciprocal sum,
\[
\sum_{\alpha\in\Z[i]\setminus\{0\}}\alpha^{-r}
=
i^{-r}\sum_{\alpha\in\Z[i]\setminus\{0\}}\alpha^{-r}.
\]
The sum can be nonzero only when $4\mid r$. The passage from the classical exponents $2n$ to the Gaussian exponents $4n$ is therefore dictated by the enlargement of the symmetry group from $\{\pm1\}$ to $\{\pm1,\pm i\}$.

\emph{The lemniscatic sine and the square-lattice elliptic curve.}
Put
\[
\omega=2\int_0^1\frac{dt}{\sqrt{1-t^4}},
\qquad
\Lambda=\omega\Z[i].
\]
The lemniscatic sine is characterized near the origin by
\[
\slf(z)=z+O(z^5),
\qquad
\slf'(z)^2=1-\slf(z)^4.
\]
For the square lattice, the normalized Weierstrass function satisfies
\[
\wp'(z)^2=4\wp(z)^3-4\wp(z),
\qquad
\wp(z)=\slf(z)^{-2}.
\]
The corresponding cubic curve is
\[
E:y^2=4x^3-4x.
\]
After the harmless change $y\mapsto y/2$, one obtains
\[
E_0:y^2=x^3-x,
\qquad
j(E_0)=1728,
\qquad
\operatorname{End}_{\mathbb C}(E_0)\simeq\Z[i].
\]
Thus the lemniscatic sine is not only a two-dimensional analogue of the circular sine: it is a local coordinate for the elliptic curve whose complex-multiplication ring is the same Gaussian lattice that underlies the rectangular products.

\emph{The Weierstrass functions and Gaussian reciprocal moments.}
The exact pole-order correspondence is
\[
\pi\cot(\pi z)\longleftrightarrow\zeta_\Lambda(z),
\qquad
\pi^2\csc^2(\pi z)\longleftrightarrow\wp(z;\Lambda).
\]
The zeta function has simple poles at the lattice points, while $\wp=-\zeta_\Lambda'$ has double poles. The Laurent expansion of $\wp$ contains the reciprocal moments of the nonzero Gaussian lattice points. Quarter-turn symmetry removes every exponent not divisible by four, leaving
\[
\sum_{\alpha\in\Z[i]\setminus\{0\}}\frac1{\alpha^{4n}}.
\]

\emph{Bernoulli--Hurwitz numbers.}
The surviving moments are normalized by the Bernoulli--Hurwitz numbers $E_n$ through
\[
\sum_{\alpha\in\Z[i]\setminus\{0\}}
\frac1{\alpha^{4n}}
=
\frac{(2\omega)^{4n}}{(4n)!}E_n,
\qquad n\geq1.
\]
Equivalently, the Laurent coefficients of the square-lattice $\wp$ function package the sequence $E_n$. They therefore occupy in the Gaussian dictionary exactly the position held by the ordinary Bernoulli numbers in the classical column. From the modern viewpoint, they are normalized Eisenstein values at the CM point $\tau=i$ .

\emph{Formal groups and complex multiplication.}
Writing $\phi(u)=\slf(u)$, the inverse lemniscatic integral is the formal logarithm of $E_0$ in the quartic coordinate, and $\phi$ is the corresponding formal exponential. If $m\in\Z[i]$ is odd and primary, the CM endomorphism $[m]$ is represented in this coordinate by a rational multiplication formula
\[
\phi(mu)=\frac{U_m(\phi(u))}{V_m(\phi(u))},
\qquad
U_m,V_m\in\Z[i][x],
\qquad
\deg[m]=N(m).
\]
The divisibility of the coefficients in these formulas is the formal-coordinate shadow of the invariant-differential identity
\[
[m]^*\omega_E=m\omega_E.
\]
This explains why the Gaussian reciprocal sums appearing in finite block calculations are naturally connected with the endomorphism theory of a CM elliptic curve, rather than being merely formal two-dimensional analogues of ordinary harmonic sums 

\emph{Finite Gaussian reciprocal sums.}
Let $\pi\in\Z[i]$ be a Gaussian prime and let $q=N(\pi)$. Since
\[
\Z[i]/(\pi)\simeq\F_q,
\]
its nonzero elements form a cyclic group of order $q-1$, and hence
\[
\sum_{x\in(\Z[i]/(\pi))^\times}x^{-m}
=
\begin{cases}
-1,&q-1\mid m,\\
0,&q-1\nmid m.
\end{cases}
\]
For a rational prime $p\equiv3\pmod4$, the prime is inert and
\[
\Z[i]/(p)\simeq\F_{p^2}.
\]
Thus complete reciprocal sums vanish modulo $p$ unless $p^2-1$ divides the exponent. The prime-power sums used above are
\[
H_j(k)=\sum_{z\in U_k}\frac1{z^j},
\qquad
U_k=\{x+iy:1\leq x,y\leq p^k,\ p\nmid(x,y)\}.
\]
The estimates
\[
H_1(k)\in p^{2k}\Op,
\qquad
H_2(k)\in p^k\Op
\]
are the Gaussian analogues of the low classical harmonic-sum congruences.

When $p\equiv1\pmod4$, one has $p=\pi\bar\pi$ and
\[
\Z[i]/(p)\simeq\F_p\times\F_p.
\]
This ring is not a field. Inverse powers are defined only on its unit group, and the two local directions above $p$ must be kept separate. This is the residue-theoretic origin of the difference between split and inert primes.

\emph{Division points and exact finite sums.}
The finite-field argument determines complete reciprocal sums modulo a Gaussian prime, but by itself it does not determine congruences modulo $\pi^2,\pi^3,\ldots$, nor sums over restricted systems of representatives. The elliptic-function viewpoint supplies additional structure. For $\alpha\in\Z[i]/p\Z[i]$, put
\[
z_\alpha=\frac{\omega\alpha}{p}.
\]
These are the $p$-division points of the square-lattice torus. Since $\wp=\slf^{-2}$,
\[
\sum_{\alpha\ne0}\wp(z_\alpha)^r
=
\sum_{\alpha\ne0}\frac1{\slf(z_\alpha)^{2r}}.
\]
The multiplication identity
\[
p^2\wp(pz)=\sum_{u\in E[p]}\wp(z+u)-C_p
\]
can be expanded at the origin to obtain exact identities for the torsion power sums. Equivalently, the coefficients of the $p$-division polynomial and Newton's identities determine the sums of powers of the $x$-coordinates $x(u)=\wp(u)$. Because the Laurent coefficients of $\wp$ are the Bernoulli--Hurwitz numbers, these torsion identities provide a systematic source of exact Gaussian reciprocal-power formulas before modular reduction. 

\emph{Frobenius, the Hasse invariant, and reduction type.}
The same elliptic curve explains geometrically the split--inert distinction. If $p=\pi\bar\pi\equiv1\pmod4$ and $\pi=a+ib$ is chosen primary, the curve $E_0$ has ordinary reduction and compatible Frobenius trace
\[
a_p(E_0)=p+1-\#E_0(\F_p)=2a.
\]
The scalar $2a$ is also the Hasse-invariant coefficient in the lemniscatic formal coordinate and the multiplier in Hurwitz's coefficient recurrences. If $p\equiv3\pmod4$, the curve has supersingular reduction and the Hasse invariant vanishes. Thus
\[
\begin{array}{c|c|c}
p\bmod4 & \text{reduction of }E_0 & \text{lemniscatic coefficient behavior}\\ \hline
1 & \text{ordinary} & \text{recurrence with multiplier }2a\\
3 & \text{supersingular} & \text{eventual vanishing modulo }p
\end{array}
\]
This gives a geometric refinement of the ring-theoretic distinction between $\F_{p^2}$ and $\F_p\times\F_p$ 

\emph{Translated square blocks.}
A complete $p^k\times p^k$ block of Gaussian units has product
\[
G_k=\prod_{z\in U_k}z.
\]
For $t\in\Z[i]$, translation by $p^kt$ gives
\[
\frac1{G_k}\prod_{z\in U_k}(z+p^kt)
=
\prod_{z\in U_k}
\left(1+\frac{p^kt}{z}\right).
\]
Its elementary symmetric coefficients are controlled by the moments $H_j(k)$. The first two estimates imply that the translated product is congruent to $G_k$ modulo $p^{3k}$. This is the two-dimensional block invariance proved in Lemma~\ref{lem:block-translation}.

The lemniscatic interpretation places this calculation in a broader hierarchy. Complete residue sums are governed at first order by cyclic finite-field groups. Higher congruences can in principle be approached through lifted reciprocal sums, multiplication formulas, torsion points, division polynomials, and Bernoulli--Hurwitz coefficients. The present theorem, however, uses only the divisibility of $H_1(k)$ and $H_2(k)$; it does not require an explicit Bernoulli--Hurwitz correction term.

\emph{Rectangular ratios and congruences.}
The coefficient $\Rect(A,B;C,D)$ compares a translated $C\times D$ rectangle with the basic $C\times D$ rectangle. In terms of rectangular factorials it is assembled by multiplicative inclusion--exclusion from five rectangles. Under simultaneous dilation by $p^k$, every large rectangle decomposes into complete $p^k\times p^k$ unit blocks together with the factors divisible by $p$. The latter reproduce the preceding scale, while the complete-block constants cancel because
\[
AB+(A-C)(B-D)-(A-C)B-A(B-D)-CD=0.
\]
Consequently,
\[
\Rect(p^kA,p^kB;p^kC,p^kD)
\equiv
\Rect(p^{k-1}A,p^{k-1}B;p^{k-1}C,p^{k-1}D)
\pmod{p^{3k}}.
\]
The Bailey-type congruence similarly separates a large $p$-block from a residual digit rectangle, while the vertical and horizontal boundary strips produce the two additional ordinary binomial factors.

The Gaussian column may therefore be summarized by the chain
\[
\begin{gathered}
\Z[i]
\longrightarrow
\slf,\ \zeta_\Lambda,\ \wp_\Lambda
\longrightarrow
E_n\\
\Downarrow\\[-2pt]
\text{finite Gaussian reciprocal moments}
\longrightarrow
\text{translated square-block invariance}\\
\Downarrow\\[-2pt]
\text{rectangular Bailey-type and}\\[-2pt]
\text{Ljunggren--Jacobsthal congruences}.
\end{gathered}
\]
The square lattice supplies the symmetry, the CM elliptic curve packages the global moments, its reductions and torsion organize the finite arithmetic, and rectangular inclusion--exclusion converts the resulting block identities into congruences for factorial ratios.

\emph{Higher congruences.}
The lemniscatic theory suggests a systematic direction beyond the modulus proved here. For lifts of reciprocal powers one may expand
\[
(a+bi+p\beta)^{-m}
=
(a+bi)^{-m}
\sum_{j\geq0}
\binom{-m}{j}
\left(\frac{p\beta}{a+bi}\right)^j.
\]
After summation, each coefficient is a higher Gaussian reciprocal-power sum. Finite-field cancellation removes most exponents, while the surviving terms may be studied through multiplication formulas or division polynomials, with Bernoulli--Hurwitz numbers supplying the corresponding Laurent coefficients. This indicates a possible route toward the first correction beyond $p^{3k}$ and toward higher-power lifts of the rectangular Bailey congruence. These refinements are not proved in the present paper; they constitute a natural continuation of the block method developed above.

\section{Consequences and further directions}\label{sec:consequences}

By Theorem \ref{thm:ljunggren} we corrected and proved the third conjecture proposed in \cite{Kalinin2025}

\begin{proposition}
    Let $p$ be a rational prime. Then$$\Rect(pA,pB;pC,pD) \equiv \Rect(A,B;C,D) \pmod{p^{3}}$$for every$$A\geq C\geq1, \qquad B\geq D\geq1,$$if and only if$$p>5, \qquad p\equiv3\pmod4.$$
\end{proposition}
\begin{proof}
    If $p>5$ and $p\equiv3\pmod4$, the congruence follows immediately by taking $k=1$ in Theorem \ref{thm:ljunggren}.We prove that every other prime must be excluded. First suppose that $p\equiv1\pmod4$. Then$$p=\pi\overline{\pi}$$for a Gaussian prime $\pi=a+bi$, where $a>b>0$. Consider the basic block$$Q_p=\{x+iy:1\leq x,y\leq p\}.$$Exactly $p$ elements of $Q_p$ are divisible by $\pi$. Moreover,$$\pi^{2}=(a^{2}-b^{2})+2abi\in Q_p,$$and therefore$$v_{\pi}\left(\prod_{z\in Q_p}z\right)\geq p+1.$$Let$$S=\{z\in Q_p:\pi\mid z\}, \qquad z=\pi q_z\quad(z\in S).$$The residues $q_z\bmod\pi$ do not form a complete residue system modulo $\pi$. Indeed,$$\sum_{z\in S}q_z = \overline{\pi}\frac{p+1}{2}(1+i) \not\equiv0\pmod\pi,$$whereas the sum of all elements of $\mathbb Z[i]/(\pi)\cong\mathbb F_p$ is zero.Consequently, one may choose $\tau\in\mathbb Z[i]$ such that$$q_z+\overline{\pi}\tau\not\equiv0\pmod\pi \qquad \text{for every }z\in S.$$After adding suitable multiples of $p$, write$$\tau=(A-1)+i(B-1), \qquad A,B\geq1.$$Then every $\pi$-divisible factor in the translated block $Q_p+p\tau$ is divisible by $\pi$ exactly once. Hence$$v_{\pi}\left(\prod_{z\in Q_p}(z+p\tau)\right)=p.$$Since$$\Rect(pA,pB;p,p) = \frac{\displaystyle\prod_{z\in Q_p}(z+p\tau)} {\displaystyle\prod_{z\in Q_p}z},$$we obtain$$v_{\pi}\bigl(\Rect(pA,pB;p,p)\bigr)\leq-1.$$On the other hand,$$\Rect(A,B;1,1)=\frac{A+iB}{1+i}$$is $\pi$-integral. Thus the two quantities cannot be congruent modulo $p^{3}$. Therefore the conjecture fails for every split prime $p\equiv1\pmod4$.It remains only to exclude $p=2$ and $p=3$. For $p=2$,$$\Rect(4,4;2,2)-\Rect(2,2;1,1) = 30-2 = 28 \not\equiv0\pmod{2^{3}}.$$For $p=3$,$$\Rect(6,6;3,3)-\Rect(2,2;1,1) = \frac{20008}{5}-2 = \frac{19998}{5},$$and$$v_3\left(\frac{19998}{5}\right)=2<3.$$Hence the conjecture also fails for $p=2$ and $p=3$.Thus the universal congruence holds precisely for$$p>5, \qquad p\equiv3\pmod4.$$
\end{proof}

As corollary of Theorem \ref{thm:ljunggren} 
 
\begin{corollary}[$p$-adic stabilization]\label{cor:padic-limit}
Under the hypotheses of Theorem~\ref{thm:ljunggren}, the sequence
\[
\left\{\Rect(p^kA,p^kB;p^kC,p^kD)\right\}_{k\geq0}
\]
is Cauchy in $\Op$, and hence converges in $\Op$.
\end{corollary}

\begin{proof}
Theorem~\ref{thm:ljunggren} gives
\[
\vp(R_k-R_{k-1})\geq3k\longrightarrow\infty.
\]
\end{proof}

The restriction $p\equiv3\pmod4$ is structural in both proofs. It ensures that $p$ is inert, that reduction modulo $p$ gives the field $\F_{p^2}$, and that divisibility of $x+iy$ by $p$ is equivalent to simultaneous divisibility of its two coordinates. When $p\equiv1\pmod4$, the prime splits in $\Z[i]$ and the residue ring has zero divisors; a corresponding theory should keep track of the two prime ideals above $p$ separately.

Theorem~\ref{thm:bailey} is a modulus-$p$ rectangular counterpart of the one-digit congruences studied by Bailey. In one dimension Bailey's factorization survives modulo $p^2$ and $p^3$, and later work gives versions for arbitrary prime powers \cite{Bailey1990,Bailey1992,DavisWebb1990,Granville1997}. It is therefore natural to ask for the strongest higher-power lift of \eqref{eq:bailey-congruence}. Such a lift may require new correction factors: the two visible binomial factors record boundary strips modulo $p$, while higher powers should also detect the first $p$-adic moments of those strips.

A second problem is global integrality. Lemma~\ref{lem:integrality} gives integrality at every inert prime, but the coefficients need not be Gaussian integers globally. Determining their denominators, especially at split primes and at the ramified prime $2$, appears to be a basic prerequisite for a fuller arithmetic theory.

\subsection{Extension to imaginary quadratic orders}\label{subsec:quadratic-orders}

The block method is not intrinsically restricted to the square lattice.  Let
$\omega$ be an imaginary quadratic algebraic integer satisfying
\begin{equation}\label{eq:omega-minimal-polynomial}
\omega^2-T\omega+N=0,
\qquad
T,N\in\Z,
\qquad
\Delta:=T^2-4N<0,
\end{equation}
and put
\[
\mathcal O_\omega:=\Z[\omega].
\]
Thus $\mathcal O_\omega$ is an order in an imaginary quadratic field, with
\[
\overline\omega=T-\omega,
\qquad
N_{\mathcal O_\omega/\Z}(x+y\omega)
=x^2+Txy+Ny^2.
\]
For $A\geq C\geq1$ and $B\geq D\geq1$, define the corresponding rectangular coefficient by
\begin{equation}\label{eq:omega-rectangular-coefficient}
\mathcal R_\omega(A,B;C,D)
:=
\frac{
\displaystyle
\prod_{n=0}^{C-1}\prod_{m=0}^{D-1}
\bigl(A+B\omega-(n+m\omega)\bigr)
}{
\displaystyle
\prod_{n=1}^{C}\prod_{m=1}^{D}(n+m\omega)
}.
\end{equation}
Equivalently, $\mathcal R_\omega$ compares two congruent parallelograms in the lattice
$\Z+\Z\omega$.  The Gaussian coefficient is recovered by taking $\omega=i$, while
$\omega=\zeta_3$ gives the corresponding coefficient over the Eisenstein integers.
The following table records the part of the Gaussian theory that admits a natural
quadratic-order analogue.

\begin{table}[ht]
\centering
\small
\renewcommand{\arraystretch}{1.18}
\begin{tabularx}{\textwidth}{>{\raggedright\arraybackslash}p{0.23\textwidth}
>{\raggedright\arraybackslash}p{0.27\textwidth}
>{\raggedright\arraybackslash}X}
\toprule
\textbf{Arithmetic datum}
& \textbf{Gaussian case}
& \textbf{General imaginary quadratic order}\\
\midrule
Order and basis
& $\Z[i]=\Z\oplus\Z i$
& $\mathcal O_\omega=\Z\oplus\Z\omega$, with $\omega$ satisfying \eqref{eq:omega-minimal-polynomial} \\
Defining equation
& $i^2+1=0$
& $\omega^2-T\omega+N=0$ \\
Discriminant
& $-4$
& $\Delta=T^2-4N<0$ \\
Conjugation
& $i\mapsto-i$
& $\omega\mapsto\overline\omega=T-\omega$ \\
Norm form
& $x^2+y^2$
& $x^2+Txy+Ny^2$ \\
Unramified inert primes
& $p\equiv3\pmod4$
& $p\nmid\Delta$ and $\left(\frac{\Delta}{p}\right)=-1$ \\
Residue field
& $\Z[i]/(p)\simeq\F_{p^2}$
& $\mathcal O_\omega/(p)\simeq\F_{p^2}$ \\
Coordinate divisibility
& $p\mid x+iy$ iff $p\mid x$ and $p\mid y$
& $p\mid x+y\omega$ iff $p\mid x$ and $p\mid y$ \\
Unit symmetry and global moments
& $\{\pm1,\pm i\}$; only exponents divisible by $4$ survive
& Usually $\{\pm1\}$; only even exponents are forced to survive.  The cases $\Delta=-4,-3$ have extra symmetry \\
CM object
& $y^2=x^3-x$, with $j=1728$
& $E_\omega=\mathbb C/(\Z+\Z\omega)$, with CM by $\mathcal O_\omega$ and singular modulus $j(\omega)$ \\
Finite unit block
& $\{x+iy:1\leq x,y\leq p^k,\ p\nmid(x,y)\}$
& $\{x+y\omega:1\leq x,y\leq p^k,\ p\nmid(x,y)\}$ \\
Rectangular ratio
& $\Rect(A,B;C,D)$
& $\mathcal R_\omega(A,B;C,D)$ \\
Expected congruences
& Theorems~\ref{thm:ljunggren} and~\ref{thm:bailey}
& Uniform inert-prime analogues of the same formal shape \\
\bottomrule
\end{tabularx}
\caption{Comparison between the Gaussian theory and the proposed extension to an arbitrary imaginary quadratic order.}
\label{tab:gaussian-general-quadratic}
\end{table}

The natural local setting is obtained by fixing a rational prime $p$ such that
\begin{equation}\label{eq:omega-inert-prime}
p>5,
\qquad
p\nmid\Delta,
\qquad
\left(\frac{\Delta}{p}\right)=-1.
\end{equation}
Then $p$ is unramified and inert, the polynomial in \eqref{eq:omega-minimal-polynomial}
remains irreducible modulo $p$, and
\[
\mathcal O_{\omega,p}:=\Z_p[\omega]
\]
is the ring of integers of the unramified quadratic extension of $\mathbb Q_p$.
In particular,
\[
\mathcal O_\omega/(p)\simeq\F_{p^2}
\]
and divisibility by $p$ is detected simultaneously in the two coordinates.  These are
precisely the structural properties used in the proofs of Lemma~\ref{lem:integrality}
and Theorem~\ref{thm:ljunggren}.

For $k\geq1$, define
\[
U_{\omega,k}
:=
\left\{
 x+y\omega:
 1\leq x,y\leq p^k,
 \quad p\nmid(x,y)
\right\}
\]
and
\[
H_{r,\omega}(k)
:=
\sum_{z\in U_{\omega,k}}z^{-r}.
\]
The following conjecture isolates the local estimate on which a uniform theory would rest.

\begin{conjecture}\label{conj:omega-moments}
Let $\omega$ satisfy \eqref{eq:omega-minimal-polynomial}, and let $p$ satisfy
\eqref{eq:omega-inert-prime}.  Then, for every $k\geq1$,
\[
H_{1,\omega}(k)\in p^{2k}\mathcal O_{\omega,p},
\qquad
H_{2,\omega}(k)\in p^k\mathcal O_{\omega,p}.
\]
\end{conjecture}

Modulo $p$, the assertion follows from the cyclicity of $\F_{p^2}^{\times}$ for the
relevant reciprocal powers.  The essential point is to prove that the extra power of
$p$ in the first moment is uniform in the choice of the quadratic order and in the
basis $(1,\omega)$.  If Conjecture~\ref{conj:omega-moments} holds, Newton's identities
and the translated-block argument of Section~\ref{sec:ljunggren-proof} give the
following expected generalization.

\begin{conjecture}
\label{conj:omega-ljunggren}
Under the hypotheses of Conjecture~\ref{conj:omega-moments}, for every $k\geq1$ and
all $A\geq C\geq1$, $B\geq D\geq1$,
\[
\begin{aligned}
&\mathcal R_\omega(p^kA,p^kB;p^kC,p^kD)
\\
&\qquad\equiv
\mathcal R_\omega
(p^{k-1}A,p^{k-1}B;p^{k-1}C,p^{k-1}D)
\pmod{p^{3k}\mathcal O_{\omega,p}}.
\end{aligned}
\]
Consequently, the sequence
\[
\left\{
\mathcal R_\omega(p^kA,p^kB;p^kC,p^kD)
\right\}_{k\geq0}
\]
converges in $\mathcal O_{\omega,p}$.
\end{conjecture}

The one-digit congruence should also persist.  Indeed, inertness gives the Frobenius
relation
\[
\omega^p\equiv\overline\omega\pmod p.
\]
For $x,y\in\F_p^\times$, one therefore obtains
\begin{align}
\prod_{t\in\F_p}(x+\omega t)
&=
\frac{\omega-\overline\omega}{\omega}\,x,
\label{eq:omega-vertical-strip}\\
\prod_{t\in\F_p}(t+\omega y)
&=
(\overline\omega-\omega)y
\label{eq:omega-horizontal-strip}
\end{align}
inside $\mathcal O_\omega/(p)$.  The constants in
\eqref{eq:omega-vertical-strip}--\eqref{eq:omega-horizontal-strip} are nonzero modulo
$p$ and their total exponents cancel in the five-rectangle inclusion--exclusion
formula, exactly as in the Gaussian proof.

\begin{conjecture}
\label{conj:omega-bailey}
Let $p$ satisfy \eqref{eq:omega-inert-prime}, and suppose
\[
1\leq\gamma\leq\alpha\leq p-1,
\qquad
1\leq\delta\leq\beta\leq p-1.
\]
Then
\[
\begin{aligned}
&\mathcal R_\omega
(pA+\alpha,pB+\beta;pC+\gamma,pD+\delta)
\\
&\quad\equiv
\mathcal R_\omega(A,B;C,D)
\mathcal R_\omega(\alpha,\beta;\gamma,\delta)
\binom{\alpha}{\gamma}^{D}
\binom{\beta}{\delta}^{C}
\pmod{p\mathcal O_{\omega,p}}.
\end{aligned}
\]    
\end{conjecture}

The analytic interpretation is less uniform than the local block argument.  For
$\omega=i$, quarter-turn symmetry selects the moments of order $4n$, and for
$\omega=\zeta_3$, sixfold symmetry selects those of order $6n$.  For a general
imaginary quadratic order the unit group is usually $\{\pm1\}$, so symmetry forces
only the odd moments to vanish.  The surviving even reciprocal moments are Eisenstein
values of the CM lattice $\Z+\Z\omega$ and occur in the Laurent expansion of its
Weierstrass function.  It is therefore natural to expect the first terms beyond the
modulus $p^{3k}$ to depend on the CM elliptic curve $E_\omega$ and on normalized
lattice moments attached to the singular modulus $j(\omega)$.

\begin{problem}\label{prob:omega-cm-corrections}
Determine the first nonzero term in
\[
\frac{
\mathcal R_\omega(p^kA,p^kB;p^kC,p^kD)
}{
\mathcal R_\omega(p^{k-1}A,p^{k-1}B;p^{k-1}C,p^{k-1}D)
}-1
\]
and express it, whenever possible, through finite reciprocal moments and the Laurent
coefficients of the Weierstrass function of $E_\omega$.
\end{problem}

There are two further structural issues.  First, the definition
\eqref{eq:omega-rectangular-coefficient} depends on the oriented integral basis
$(1,\omega)$, not only on the abstract order $\mathcal O_\omega$.  A basis-free theory
should replace coordinate rectangles by parallelograms in fractional ideals.  Second,
when $p$ splits, the rational ideal $(p)$ factors into two prime ideals and the universal
congruence modulo $p^{3k}$ should not be expected.  The appropriate replacement should
be a pair of local congruences, one at each prime above $p$, with correction factors
reflecting the two CM components of Frobenius.

\begin{problem}\label{prob:omega-ideal-formulation}
Construct a basis-independent rectangular coefficient attached to an oriented ideal
lattice, and formulate separate local congruences at the two prime ideals above an
unramified split prime.  Determine which parts of Conjectures~\ref{conj:omega-moments}--\ref{conj:omega-bailey} survive after this ideal-theoretic reformulation.
\end{problem}

\subsection{Gamma-function representations}
The defining products can be rewritten in terms of Euler's Gamma function by repeated use of
$\Gamma(z+1)=z\Gamma(z)$; see, for example, \cite[Chapter~1]{AndrewsAskeyRoy1999}.

\begin{proposition}[Gamma representations]\label{prop:gamma-representations}
For $A\geq C\geq1$ and $B\geq D\geq1$,
\begin{align}
\Rect(A,B;C,D)
&=
\prod_{y=B-D+1}^{B}
\frac{\Gamma(A+1+iy)}{\Gamma(A-C+1+iy)}
\prod_{y=1}^{D}
\frac{\Gamma(1+iy)}{\Gamma(C+1+iy)},
\label{eq:gamma-horizontal}\\
&=
\prod_{x=A-C+1}^{A}
\frac{\Gamma(B+1-ix)}{\Gamma(B-D+1-ix)}
\prod_{x=1}^{C}
\frac{\Gamma(1-ix)}{\Gamma(D+1-ix)}.
\label{eq:gamma-vertical}
\end{align}
\end{proposition}

\begin{proof}
For fixed $y$,
\[
\prod_{x=A-C+1}^{A}(x+iy)
=
\frac{\Gamma(A+1+iy)}{\Gamma(A-C+1+iy)},
\qquad
\prod_{x=1}^{C}(x+iy)
=
\frac{\Gamma(C+1+iy)}{\Gamma(1+iy)}.
\]
Multiplying over the relevant values of $y$ proves \eqref{eq:gamma-horizontal}. Alternatively, for fixed $x$ one factors $i$ from each term and obtains
\begin{align*}
\prod_{y=B-D+1}^{B}(x+iy)
&=
i^D\frac{\Gamma(B+1-ix)}{\Gamma(B-D+1-ix)},\\
\prod_{y=1}^{D}(x+iy)
&=
i^D\frac{\Gamma(D+1-ix)}{\Gamma(1-ix)}.
\end{align*}
The factors $i^{CD}$ cancel between numerator and denominator, giving \eqref{eq:gamma-vertical}.
\end{proof}

These formulas provide a first analytic continuation in the upper parameters and show that the coefficients are hypergeometric in each lower parameter. More concretely, extending the convention by $\Rect(A,B;0,D)=\Rect(A,B;C,0)=1$, one has the multiplicative recurrences
\begin{align}
\frac{\Rect(A,B;C,D)}{\Rect(A,B;C-1,D)}
&=
\frac{\displaystyle\prod_{y=B-D+1}^{B}(A-C+1+iy)}
{\displaystyle\prod_{y=1}^{D}(C+iy)},
\label{eq:C-recurrence}\\
\frac{\Rect(A,B;C,D)}{\Rect(A,B;C,D-1)}
&=
\frac{\displaystyle\prod_{x=A-C+1}^{A}(x+i(B-D+1))}
{\displaystyle\prod_{x=1}^{C}(x+iD)}.
\label{eq:D-recurrence}
\end{align}
They are two-dimensional analogues of the successive-quotient formula
\[
\frac{\binom{N}{K}}{\binom{N}{K-1}}
=
\frac{N-K+1}{K}.
\]

Because \eqref{eq:gamma-horizontal} and \eqref{eq:gamma-vertical} contain finite strings of Gamma values along the Gaussian lattice, Barnes's double Gamma function is a natural candidate for compressing them into a quotient of a bounded number of special-function factors \cite{Barnes1899,Barnes1901}. This viewpoint suggests the following problems.

\begin{problem}[Barnes continuation]\label{prob:barnes}
Construct a canonical Barnes double-Gamma representation with periods $1$ and $i$. Determine whether it gives a meromorphic continuation of $\Rect(A,B;C,D)$ in all four parameters that is compatible with both lattice directions and with the arithmetic normalization used in this paper.
\end{problem}

\begin{problem}[Pascal and complementary-rectangle identities]\label{prob:pascal}
Determine whether there is a natural normalization
\[
\widetilde{\Rect}(A,B;C,D)
=
N(A,B;C,D)\Rect(A,B;C,D)
\]
for which the multiplicative recurrences \eqref{eq:C-recurrence}--\eqref{eq:D-recurrence} can be converted into a local additive Pascal-type identity. Likewise, decide whether a normalized complementary-rectangle symmetry exists under
\[
(C,D)\longleftrightarrow(A-C,B-D),
\]
possibly with an explicit boundary correction factor.
\end{problem}

\begin{problem}[Vandermonde, generating functions, and inversion]\label{prob:vandermonde}
Find, or rule out, two-dimensional analogues of the Chu--Vandermonde convolution, the binomial theorem, and binomial inversion. In particular, study bivariate generating functions of the form
\[
\mathcal G_{A,B}(u,v)
=
\sum_{C=0}^{A}\sum_{D=0}^{B}
\widetilde{\Rect}(A,B;C,D)u^Cv^D
\]
and ask whether an appropriate normalization admits a product, determinant, or hypergeometric representation. A satisfactory identity should explain geometrically how a large rectangle decomposes into smaller rectangles and should recover the ordinary one-dimensional identity on a suitable boundary specialization.
\end{problem}

A combinatorial or representation-theoretic interpretation would be especially valuable. It could explain which normalization is canonical, why the boundary factors in Theorem~\ref{thm:bailey} appear, and whether positivity or integrality survives after passing from ordinary binomial coefficients to Gaussian rectangular ones.

\section*{Acknowledgements}
The author thanks Nikita Kalinin , Mikhail Shkolnikov , Higinio Serrano ,  and Ernesto Lupercio for valuable suggestions and corrections.

\nocite{*}
\printbibliography

\end{document}